\documentclass[]{amsart}
\usepackage[utf8]{inputenc}
\usepackage{amsmath,amssymb,amsthm,mathtools}
\usepackage{graphicx}
\usepackage{verbatim}
\usepackage{xcolor}

\newtheorem{defn}{Definition}
\newtheorem{thm}{Theorem}

\newtheorem{cor}{Corollary}
\newtheorem{prop}{Proposition}
\newtheorem{lem}{Lemma}

\newtheorem*{property*}{Property}
\newtheorem*{theorem*}{Theorem}

\newtheorem*{problem*}{Problem}

\theoremstyle{remark}
\newtheorem*{sketch*}{Sketch of Proof}

\title{On non-local approximation properties of the binomial power functions $(1+x^q)^r$}

\author{Brock Erwin}
\address{Longwood University, Farmville, VA 23909, USA}
\email{brock.erwin@live.longwood.edu}

\author{Jeff Ledford}
\address{Department of Mathematics and Computer Science, Longwood University, Farmville, VA 23909, USA}
\email[Corresponding author]{ledfordjp@longwood.edu}

\author{Kira Pierce}
\address{Longwood University, Farmville, VA 23909, USA}
\email{kira.pierce@live.longwood.edu}

\thanks{The authors were supported by the PRISM program at Longwood University.}



\begin{document}

\begin{abstract}
\noindent This note mainly concerns the binomial power function, defined as $(1+x^q)^{r}$.  We construct systems of polynomials related to non-local approximation, which allows us to establish the density results on $C[a,b]$, where $a,b\in\mathbb{R}$.  As a corollary, we show that scattered translated of power functions and certain related functions are dense in the function spaces $L^p([a,b])$, for $1\leq p <\infty$.
\end{abstract}

\maketitle

\section{Introduction}

In this paper, we consider the approximation set
\[
S(\varphi, X):= \left\{\sum_{j=1}^{N}a_j\varphi(x-x_j): N\in\mathbb{N}, a_j\in\mathbb{R}, x_j\in X \right\},
\]
where $X$ is an appropriately chosen `scattered' sequence.  We seek conditions on $\varphi$ such that the following theorem is true.
\begin{theorem*}
Suppose $f\in C[a,b]$.  For any $\varepsilon >0$, there exists $s\in S(\varphi,X)$, such that
\[
\|f-s \|_{L_\infty}<\varepsilon.
\]
\end{theorem*}

Approximation sets $S$ are encountered frequently in approximation theory.  A typical result uses the Fourier transform to pass on properties of the sequence $(a_j)$ onto $s\in S$ in the $L_2$ norm.  Usually this is done fixing the kernel, see for instance [scattered interpolation papers and pick out the various kernels].  

Our approach differs in that we will forgo the use of the Fourier transform and instead seek to leverage the series representation of power functions.  This approach has appeared in the literature before, notably \cite{Powell} provides results for the \emph{Hardy multiquadric} $\sqrt{1+x^2}$, while \cite{Ledford1} and \cite{Ledford2} provide similar results for the Poisson kernel $(1+x^2)^{-1}$ and \emph{general multiquadric} $(1+x^2)^{k-1/2}$, respectively.  Our approach unifies the previous ones and provides novel examples related to \emph{binomial power functions} $(1+x^q)^r$, where $q\in \mathbb{N}$ and $r\in\mathbb{R}\setminus\mathbb{N}_0$.

The rest of this paper is organized as follows.  The second section contains various definitions and basic facts necessary to the sequel.  The third section provides the main result and a short proof, while the final section is devoted to examples.

\section{Definitions and Basic Facts}

Throughout the sequel, we denote by $\mathbb{N}_{0}$ the collection of non-negative integers.
We denote the space of polynomials of degree at most $n$ by $\Pi_n$ and let $\Pi:=\bigcup_{n\in\mathbb{N}_0}\Pi_n$.

Our first few definitions are devoted to the sequence $X$. 

\begin{defn}
A sequence of real numbers, denoted $X$, is said to be $\delta$-separated if
\[
\inf_{\overset{x,y\in X}{x\neq y}}|x-y|= \delta >0
\]
\end{defn}
\noindent Note that a $\delta$-separated sequence must be countable.  Taking intervals of length $\delta/3$ centered at each point in $X$ yields disjoint intervals, each of which contains a rational number $r$.  Letting a member of $X$ correspond to the number $r$ which is in the same interval shows that the set $X$ is at most countable.  This allows us to index $X$ with the integers.

\begin{defn}
A sequence $\{x_j\}\subset \mathbb{R}$ is \emph{scattered} if it is $\delta$-separated for some $\delta>0$ and satisfies 
\[
\lim_{j\to\pm\infty}x_j=\pm\infty.
\]
\end{defn}
\noindent Throughout the remainder of the paper we let $X=\{x_j\}_{j\in\mathbb{Z}}$ be a fixed but otherwise arbitrary scattered sequence.
Of use to us will be the following notion.  
\begin{defn}
 $Y\subset \mathbb{R}$ is a positive (negative) \emph{doubling sequence} if
 \begin{enumerate}
     \item $y_1>0$ \quad$(y_1<0)$, and
     \item $y_{j+1}\geq 2y_{j} \quad(y_{j+1}\leq 2y_{j})  ; j\in\mathbb{N}$.
 \end{enumerate}
\end{defn}
 
 \begin{lem}\label{doubling lemma}
 Every scattered sequence $X$ contains both a positive and negative doubling subsequence. Additionally, for any $M>0$, we can find a doubling subsequence $Y$ such that $|y_1|>M$.
 \end{lem}
\begin{proof}
Let $M>0$.  Since $\displaystyle\lim_{j\to\infty}x_j=\infty$, there exists $J\in\mathbb{Z}$ such that $x_j>M$ for $j\geq J$.  Since $X$ is $\delta$-separated, we can find the smallest such $x_j$, this we call $y_1$.  Now we can repeat this procedure for $M=2y_1$ to produce $y_2$.  Continuing on in this fashion produces a positive doubling subsequence $Y:=(y_j)$.  A negative doubling subsequence is produced in an analogous manner. 
\end{proof}

\begin{lem}\label{Vandermonde_bound}
Suppose that $X$ is a scattered sequence and that $(y_j)\subset X$ is a doubling subsequence, then for all $i\in\mathbb{N}$
\[
\left| \prod_{j\neq i}\left[1-\dfrac{y_i}{y_j}\right]^{-1} \right| \leq 4.
\]
\end{lem}

\begin{proof}
Since $y_i/y_j>0$ for both positive and negative doubling subsequence, it is enough to consider only the positive doubling subsequence.  Fix $i\in\mathbb{N}$.  For $1\leq j < i$, we have $|(1-y_i/y_j)^{-1}|\leq 1$, hence
\[
\left|\prod_{j\neq i}\left[1-\dfrac{y_i}{y_j}\right]^{-1}\right| \leq  \prod_{j= i+1}^{\infty}\left[1-\dfrac{y_i}{y_j}\right]^{-1}.
\]
To see the bound, note that for $j>i$:
\[
\dfrac{y_i}{y_j} \leq -\ln\left( 1-\dfrac{y_i}{y_j}\right) \leq 2\ln(2)\dfrac{y_i}{y_j},
\]
which follows from the convexity of the logarithm and the fact that $y_{i+1}\geq 2y_{i}$.  Hence we have
\begin{align*}
\prod_{j= 1}^{\infty}\left[1-\dfrac{y_i}{y_{i+j}}\right]^{-1} &= \exp\left[-\sum_{j=1}^{\infty}\ln(1-y_i/y_{i+j}) \right] \\
&\leq \exp\left[ 2\ln(2) \sum_{j=1}^{\infty}y_i/y_{i+j} \right]\\
&\leq \exp\left[ 2\ln(2) \sum_{j=1}^{\infty}2^{-j}      \right]=4.
\end{align*}
\end{proof}


For a fixed $\varphi$, $X$, and $n\in\mathbb{N}$ we let
\[
S_n(\varphi,X):=\left\{ \sum_{j=1}^{n}a_j\varphi(x-x_j) : a_j\in\mathbb{R}, x_j\in X  \right\}
\]
and set $S(\varphi,X):=\bigcup_{n\in\mathbb{N}}S_n(\varphi,X)$.  When there is no confusion, we will drop the dependence on $\varphi$ and $X$.

Our strategy is to expand $\varphi(x-y)$ as a Taylor series in $y^{-1}$, the coefficients of which are polynomials in $x$.  Suppose that translates of $\varphi$ enjoy the representation
 \begin{equation}\label{translate_formula}
     \varphi(x-y) = F(y)\sum_{k=0}^{\infty}\dfrac{A_{k}(x)}{y^k},
 \end{equation}
 where $(A_k)\subset\Pi$.  We will call a function $\varphi$ which has such a representation \emph{admissible} provided $F(y)$ is eventually non-zero.

Our next few results concern alternant matrices. 

We begin with the $N\times N$ Vandermonde system associated to a doubling sequence $Y$, $V_N\mathbf{c}=\mathbf{e}_N\in\mathbb{R}^N$, where
\[
V_N:=\begin{bmatrix}
y_{j}^{-(i-1)}
\end{bmatrix}_{1\leq i,j\leq N}.
\]
The solution  can be found using Cramer's Rule, namely
\begin{equation}\label{Vandermonde_coefficients}
    c_i = y_{i}^{N-1}\prod_{j\neq i}\left[1-\dfrac{y_i}{y_j}\right]^{-1}.
\end{equation}
In light of Lemma \ref{Vandermonde_bound}, we have
\begin{equation}
    c_i = O(y_i^{N-1}).
\end{equation}

For admissible $\varphi$ and a doubling sequence $Y$, we get the related system
\begin{equation}\label{Vandermonde system}
\left[F(y_j)y_j^{-i+1}\right]_{1\leq i,j\leq N}\mathbf{\tilde{a}} = \mathbf{e}_N,
 \end{equation}
where $F$ is defined in \eqref{translate_formula}.  Using Cramer's Rule, we have
\begin{equation}\label{coefficients}
     \tilde{a}_{N,i} = \dfrac{c_i}{F(y_i)},
\end{equation}
 where $c_i$ is defined in \eqref{Vandermonde_coefficients}.


\section{Main Result}
  
We are in position to efficiently prove our main result which allows us to pass the density of $\Pi$ in $C[a,b]$ to $S(\varphi,X)$.

 \begin{thm}\label{main 1}
 Suppose that $X$ is a scattered sequence and that $\varphi$ is admissible.  If $(A_k:k\in\mathbb{N}_0)$ defined in \eqref{translate_formula} is a basis for $\Pi$, then for any $f\in C[a,b]$ and $\varepsilon>0$, there exists $s\in S(\varphi,X)$ such that 
 \[
 \| f-s \|_{L_\infty}<\varepsilon.
 \]
 \end{thm}

\begin{proof}
In light of the Stone-Weierstrass Theorem, it is enough to consider $f\in \Pi$.  Using \eqref{translate_formula}, we need only find $\tilde{s}_n\in S_n(\varphi,X)$ such that $\tilde{s}_n(x)=A_{n-1}(x) + O(y_1^{-1})$.  For $(y_j)\subset X$, we have
\begin{align*}
\sum_{j=1}^{n}a_j\varphi(x-y_j) =& \sum_{j=1}^{n}a_j F(y_j)\sum_{k=0}^{\infty}\dfrac{A_{k}(x)}{y_{j}^{k}}\\
=&\sum_{k=0}^{n-1}\left(\sum_{j=1}^{n}a_jF(y_j)y_{j}^{-k} \right)A_{k}(x)+ \sum_{j=1}^{n}\sum_{k=n}^{\infty}a_j F(y_j)\dfrac{A_{k}(x)}{y_{j}^{k}}.
\end{align*}
The first sum is a Vandermonde system, hence if $(y_k)\subset X$ is a doubling sequence and $(a_j)=(\tilde{a}_{n,j})$ is chosen as in \eqref{coefficients} then we have
\[
\sum_{j=1}^{n}\tilde{a}_{n,j}\varphi(x-y_j) = A_{n-1}(x)+ O(y_{1}^{-1}).
\]
Since $(A_k)$ is a basis for $\Pi$ and $y_1$ may be chosen arbitrarily large, the proof is complete.
\end{proof}

Applying H\"{o}lder's inequality yields the following. 
\begin{cor}
Let $p\geq 1$ and suppose that $X$ is a scattered sequence and that $\varphi$ is admissible.  If $(A_k:k\in\mathbb{N}_0)$ is a basis for $\Pi$, then for any $f\in C[a,b]$ and $\varepsilon>0$, there exists $s\in S(\varphi,X)$ such that 
 \[
 \| f-s \|_{L_p}<\varepsilon.
 \]
\end{cor}

It may happen that $(A_k:k\geq 0)$ fails to be a basis for $\Pi$ while $(A_k:k\geq K)$ is a basis for $\Pi$, in this situation the proof can be amended above by splitting the first $K+n+1$ terms from the rest
\begin{align*}
\sum_{j=1}^{n}a_j\varphi(x-y_j) =& \sum_{j=1}^{n}a_j F(y_j)\sum_{k=0}^{\infty}\dfrac{A_{k}(x)}{y_{j}^{k}}\\
=&\sum_{k=0}^{K+n-1}\sum_{j=1}^{K+n}a_jF(y_j)y_{j}^{-k} A_{k}(x)+ \sum_{j=1}^{n}\sum_{k=K+n}^{\infty}a_j F(y_j)\dfrac{A_{k}(x)}{y_{j}^{k}},
\end{align*}
now letting $a_j=\tilde{a}_{K+n-1,j}$ produces $A_{K+n-1}(x)+O(y_{1}^{-1})$.  Since $(A_{K+n-1}:n\in\mathbb{N})$ is a basis for $\Pi$, the conclusion of Theorem \ref{main 1} still holds.  We summarize this in the following.

\begin{thm}\label{main 2}
Suppose that $X$ is a scattered sequence and that $\varphi$ is admissible.  If there exists $K\in\mathbb{N}_0$ such that $(A_k:k\geq K)$ is a basis for $\Pi$, then for any $f\in C[a,b]$ and $\varepsilon>0$, there exists $s\in S(\varphi,X)$ such that 
 \[
 \| f-s \|_{L_\infty}<\varepsilon.
 \]
\end{thm}

\begin{cor}
Let $p\geq 1$ and suppose that $X$ is a scattered sequence and that $\varphi$ is admissible.  If there exists $K\in\mathbb{N}_0$ such that $(A_k:k\geq K)$ is a basis for $\Pi$, then for any $f\in C[a,b]$ and $\varepsilon>0$, there exists $s\in S(\varphi,X)$ such that 
 \[
 \| f-s \|_{L_p}<\varepsilon.
 \]
\end{cor}


\section{Examples}

Theorems \ref{main 1} and \ref{main 2} require that we show the sequence of polynomials $(A_k)$ defined in \eqref{translate_formula} are a basis for $\Pi$.  Thus the bulk of the work in examples is justifying this.  We begin with the example that motivated the treatment above.
\subsection{Binomial power functions}
Let $q\in\mathbb{N}$ and $r\in\mathbb{R}\setminus\{0\}$, then the binomial power function with shape parameter $c>0$ is
\[
\varphi(x)=(c+x^q)^r.
\]
For simplicity, we will often let $c=1$.  We begin with the special case $q\in 2\mathbb{N}$.  
\begin{lem}\label{lemma gen lead}
Let $r\in\mathbb{R}\setminus\{0\}$, $q\in 2\mathbb{N}$, and for $k\in\mathbb{N}_0$ suppose $A_k$ is defined by \eqref{translate_formula}.  Then $A_k(x)$ is given by 
\[
A_k(x)= (-1)^k\binom{qr}{k}x^k + \text{ lower order terms}.
\]
\end{lem}
\begin{proof}
For $y$ large enough, we have
\begin{align*}
\varphi(x-y) =& \left( c + (x-y)^q\right)^r    \\
=& y^{qr}\left( cy^{-q}+ \left(1-\dfrac{x}{y}\right)^q \right)^r\\
=&y^{qr}\sum_{j=0}^{\infty}\sum_{k=0}^{\infty}(-1)^{k}\binom{r}{j}\binom{q(r-j)}{k}c^jx^{k}y^{-(qj+k)}\\
=&y^{qr}\sum_{j=0}^{\infty}\sum_{k=qj}^{\infty}(-1)^{k}\binom{r}{j}\binom{q(r-j)}{k-qj}c^j x^{k-qj}y^{-k}\\
=&y^{qr}\sum_{k=0}^{\infty}\left(\sum_{j=0}^{\lfloor k/q \rfloor}(-1)^{k}\binom{r}{j}\binom{q(r-j)}{k-qj}c^j x^{k-qj}\right)y^{-k}.
\end{align*}
This means that 
\begin{equation}\label{A}
A_k(x)=\sum_{j=0}^{\lfloor k/q \rfloor}(-1)^{k}\binom{r}{j}\binom{qr-qj}{k-qj}c^j x^{k-qj},
\end{equation}
which is the desired result.
\end{proof}

\begin{cor}\label{gen case}
Suppose $r\in\mathbb{R}\setminus\{0\}$ and $q\in2\mathbb{N}$ satisfy $qr\notin \mathbb{N}$, then $(A_k)$ is a basis for $\Pi$.
\end{cor}
\begin{proof}
If $qr\notin\mathbb{N}$, then the leading coefficient in \eqref{A} cannot be $0$.
\end{proof}

\begin{cor}
Suppose $qr\in\mathbb{N}$. Then for $A_k$ defined in \eqref{translate_formula}, we have
\[
A_k(x) = \begin{cases}
(-1)^k\binom{qr}{k}x^k+ \text{ lower order terms}, &  0\leq k < q\lceil r \rceil\\  
\binom{r}{\lceil r \rceil}\binom{qr-q\lceil r \rceil }{k-q\lceil r \rceil}c^{\lceil r \rceil }x^{k-q\lceil r \rceil} +\text{ lower order terms}, & k\geq q\lceil r \rceil,
\end{cases} 
\]
hence $(A_k:k\geq q\lceil r \rceil )$ is a basis for $\Pi$.
\end{cor}
  
\begin{proof}
The formula For $0\leq k< q\lceil r \rceil$, follows from the fact that $\binom{qr}{k}\neq 0$ for these $k$.  In order to see the formula for $k\geq q\lceil r \rceil$, we note that since $r\notin \mathbb{N}_0$, $\lceil r \rceil> r$, so that the second binomial coefficient will be 0 whenever the index is less than $r$, whence \eqref{A} reduces to
\begin{equation}\label{special_case}
A_k(x)=(-1)^k\sum_{i=\lceil r \rceil}^{\lfloor k/q \rfloor}\binom{r}{i}\binom{qr-qi}{k-qi}c^ix^{k-qi}.
\end{equation}
\end{proof}

It is natural to ask what happens if $q\in\mathbb{N}$ is odd.  The argument given above is invalid for a general $r\in\mathbb{R}$ when $q$ is odd.  However, if $q$ is odd and $qr\in \mathbb{N}$, then we more or less recover \eqref{special_case}:
\[
A_k(x) = (-1)^{k+qr} \sum_{i=\lceil r \rceil}^{\lfloor k/q \rfloor}\binom{r}{i}\binom{qr-qi}{k-qi}c^ix^{k-qi}.
\]
The difference is that we must use a negative doubling sequence that is contained in our scattered sequence.  Hence we have the following.
\begin{lem}
Suppose $r\in\mathbb{R}\setminus\{0\}$ and $q\in\mathbb{N}$ is odd.  If $qr\in \mathbb{N}$, then $(A_k: k\geq q\lceil r \rceil)$ defined in \eqref{translate_formula} is a basis for $\Pi$.
\end{lem}

We summarize these results in the following.

\begin{prop}
Suppose that $\varphi$ is a binomial power function with parameters $q\in\mathbb{N}$ and $r\in \mathbb{R}\setminus\{0\}$.  Then $(A_k:k\geq K)$ defined in \eqref{translate_formula} is a basis for $Pi$ if
\begin{enumerate}
    \item $q\in 2\mathbb{N}$ and $qr\notin \mathbb{N}$ and $K=0$, or
    \item $q\in 2\mathbb{N}$, $qr\in\mathbb{N}$ and $K=q\lceil r \rceil$, or
    \item $q\in (2\mathbb{N}-1)$, $qr\in\mathbb{N}$ and $K=q\lceil r \rceil$. 
\end{enumerate}
\end{prop}

Note that in each of the cases above, $\varphi$ is admissible with $F(y) = |y|^{qr}$.  We end this section by noting that this class of examples subsumes those found in earlier works.  The Hardy multiquadric found in \cite{Powell} is $q=2$ and $r=1/2$, while the examples in \cite{Ledford1} and \cite{Ledford2} both have $q=2$, with $r=-1$ and $r=k-1/2$, respectively.


\subsection{Arctangent}

The examples in this section are related to the binomial power functions by differentiation.  We will begin with $\arctan(x)$, which satisfies
\[
\arctan(x-y) = -\dfrac{\pi}{2}+\sum_{k=1}^{\infty}\dfrac{B_k(x)}{y^k},
\]
hence we will let $\varphi(x) = \arctan(x)+\pi/2$.  Since the derivative of $\varphi$ is the Poisson kernel, $(B_k)$ can be calculated easily in terms of $(\tilde{A}_k)$ found in \cite{Ledford1}, for $k\in\mathbb{N}$ we have
\begin{align*}
B_k(x)&=  \dfrac{\tilde{A}_{k-1}(x)}{k}\\
&=x^{k-1}+\text{ lower order terms}.
\end{align*}
We can use \eqref{A}, to generate the polynomial if we need the lower order terms.  Hence $\varphi$ is admissible with $F(y)=y^{-1}$, we have 
\[
\varphi(x-y) = y^{-1}\sum_{k=0}\dfrac{A_k(x)}{y^k},
\]
with $A_k=B_{k+1}$.  Hence Corollary \ref{gen case} provides that $(A_k)$ is a basis for $\Pi$.

We can combine examples using the Cauchy product.  For instance, 
\begin{equation}\label{gen arctan}
    \varphi(x)= \left(1+x^q\right)^r \left(\arctan(x)+\dfrac{\pi}{2}\right),
\end{equation}
leads to the following.
\begin{lem}\label{lem gen arctan}
For $\varphi$ in \eqref{gen arctan}, $q\in2\mathbb{N} $, $r\in\mathbb{R}\setminus\{0\}$, and $qr\notin\mathbb{N}$,
\[
\varphi(x-y) =: y^{qr-1}\sum_{k=0}^{\infty}\dfrac{C_k(x)}{y^k},
\]
we have
\[
    C_k(x) =(-1)^{k}\binom{qr-1}{k}x^{k}+ \emph{lower order terms};     \quad k\geq 1.
\]
Hence $\varphi$ is admissible with $F(y)=y^{qr}$ and $(A_k:k\geq 1)$ is a basis for $\Pi$. 
\end{lem}
Before proving this we need the following summation formula.
\begin{lem}
Suppose that $u\in\mathbb{R}\setminus\{0\}$ and $k\in\mathbb{N}_0$, then
\begin{equation}\label{sum formula}
\sum_{j=0}^{k} (-1)^j\binom{u}{j} = (-1)^k\binom{u-1}{k},
\end{equation}
where $\binom{u}{j}$ is the general binomial coefficient.
\end{lem}
\begin{proof}
We fix $u\in\mathbb{R}\setminus\{0\}$ and induct on $k\in\mathbb{N}_0$.  When $k=0$, there is nothing to show since both sides are $1$.  Now suppose that the formula holds for some $k\geq0$ and consider 
\begin{align*}
 \sum_{j=0}^{k+1} (-1)^j\binom{u}{j} &=  (-1)^{k+1}\binom{u}{k+1} + \sum_{j=0}^{k} (-1)^j\binom{u}{j} \\
 &=(-1)^{k+1}\binom{u}{k+1} + (-1)^k\binom{u-1}{k}\\
 &=(-1)^{k+1}\left(\dfrac{u(u-1)\cdots(u-k)}{(k+1)!} - \dfrac{(u-1)(u-2)\cdots(u-k)}{k!}  \right)\\
 &=(-1)^{k+1}\dfrac{(u-1)(u-2)\cdots(u-k)}{(k+1)!}\left(u-1-k \right)\\
 &=(-1)^{k+1}\binom{u-1}{k+1},
\end{align*}
which is the desired formula.
\end{proof}

\begin{proof}[Proof of Lemma \ref{lem gen arctan}]
All we need to do is to use the series representation for each piece of the product.  This yields
\[
C_k(x) = \sum_{j=0}^{k} (-1)^j\binom{qr}{j}x^k + \text{ lower order terms}.
\]
Now \eqref{sum formula} provides the desired result.
\end{proof}

\subsection{Logarithms}

Due to the relative simplicity of the previous example, it is natural to investigate $\ln(1+x^2)$.  However, this turns out to be less straightforward and requires an updated version of the system \eqref{Vandermonde system} in order to prove a result similar to Theorem \ref{main 1}.  For our purposes, we will define 
\[
\varphi(x)=x^{-1}\ln(1+x^2),
\]
which leads to the series representation
\[
\varphi(x-y) =: \ln|y|\sum_{j=1}^{\infty}\dfrac{A_j(x)}{y^{j}} + \sum_{k=2}^{\infty}\dfrac{B_k(x)}{y^k},
\]
where $(A_j)$ and $(B_k)$ are defined below.

\begin{lem}\label{log lemma}
For $A_j$ and $B_k$ defined above, we have
\begin{align*}
A_{j}(x) & =-2x^{j-1}; \quad j\geq 1,\\
B_k(x) & =\left( \sum_{n=1}^{k-1}\dfrac{2}{n} \right)x^{k-1}+\emph{lower order terms} ; \quad k\geq 2.
\end{align*}
\end{lem}
\begin{proof}
The main tool here is again the Cauchy product as well as the results of \cite{Ledford1}. We have
\begin{align*}
 \dfrac{d}{dy}\ln(1+(x-y)^2) &= \dfrac{-2(x-y)}{1+(x-y)^2}   \\ 
 &= -2(x-y)\sum_{j=0}^{\infty}\dfrac{\tilde{A}_j(x)}{y^{j+2}},
\end{align*}
where $\tilde{A}_j$ are the polynomials corresponding to the Poisson kernel, which can be found in \cite{Ledford1} or by using \eqref{A}.  Regrouping yields
\begin{align*}
\dfrac{d}{dy}\ln(1+(x-y)^2) &= \dfrac{2\tilde{A}_0(x)}{y}+\sum_{j=2}^{\infty} \dfrac{2\tilde{A}_{j-1}(x)-2x\tilde{A}_{j-2}(x)}{y^{j}}\\
&= \dfrac{2}{y} + \sum_{j=2}^{\infty}\dfrac{C_j(x)}{y^j},
\end{align*}
where $C_j(x) = 2x^{j-1}$+ lower order terms.
Now integrating yields
\[
\ln(1+(x-y)^2) = 2\ln|y| - \sum_{j=1 }^{\infty}\dfrac{C_{j+1}(x)}{jy^j} 
\]
and since 
\[
(x-y)^{-1} = -y^{-1}\sum_{j=0}^{\infty}\dfrac{x^j}{y^j},
\]
the Cauchy product can be employed here.  This produces
\begin{align*}
\varphi(x-y) &= \ln|y|\sum_{j=1}^{-\infty}\dfrac{-2x^{j-1}}{y^j} + \sum_{j,k=1}^{\infty}\dfrac{x^{j-1}C_{k+1}}{ky^{j+k}}\\
&=\ln|y|\sum_{j=1}^{-\infty}\dfrac{-2x^{j-1}}{y^j} + \sum_{k=2}^{\infty}\dfrac{B_k(x)}{y^k},
\end{align*}
where 
\[
B_k(x) = \left(\sum_{j=1}^{k-1} \dfrac{2}{n}\right) x^{k-1}+ \text{ lower order terms},
\]
since it is clear above that $A_j(x) = -2x^{j-1}$, the proof is complete.
\end{proof}

So the difficulty with this example is that our $\varphi$ is not admissible with a single function $F(y)$, instead it is a combination of two distinct series.  Attempting to isolate $A_{N}$ as in the proof of Theorem \ref{main 1} leads us to a more general $(2N-1)\times (2N-1)$ alternant system
\begin{equation}\label{log system}
\begin{bmatrix}
y_1^{-2} & y_2^{-2} & \cdots & y_{2N-1}^{-2}\\
\vdots & \vdots & &\vdots\\
y_1^{-N}& y_2^{-N}& \cdots & y_{2N-1}^{-N}\\
y_{1}^{-1}\ln(y_1) & y_2^{-1}\ln(y_2)&\cdots& y_{2N-1}^{-1}\ln(y_{2N-1})\\
\vdots & \vdots & &\vdots\\
y_{N}^{-1}\ln(y_1) & y_2^{-N}\ln(y_2)&\cdots& y_{2N-1}^{-N}\ln(y_{2N-1})
\end{bmatrix}
\begin{bmatrix}
\tilde{a}_1\\
\tilde{a}_2\\
\vdots\\
\tilde{a}_{2N-1}
\end{bmatrix}
=
\begin{bmatrix}
0\\
\vdots\\
0\\
1
\end{bmatrix},
\end{equation}
where $N\geq 2$.  We will be able to recover $A_{N-1}$ provided we can show two things.  The first is that this system always has a solution for a suitable chosen doubling sequence $Y=(y_j)$.  The second piece of information we need is the growth rate of the $\tilde{a}_i$.  The solution of this system is related to the rational interpolation problem for logarithmic data samples.  Solvability of the rational interpolation problem is characterized by the invertibility of a certain Loewner matrix [get reference], which we would like to avoid here.  We find it more convenient to derive the growth rate of the components of the solution $\tilde{a}_i$ by appealing to properties of the logarithm and its derivatives.  Before we do this however, we will need the following general framework associated to alternant matrices.

Suppose we have a set of continuous functions
\[
G:=\{g_1,g_2,\dots,g_N  \},
\]
where for $1\leq j\leq N$, $g_j:I\to\mathbb{R}$ for some interval $I\subset\mathbb{R}$.  Now define
\[
\mathcal{G}:=\text{span}\{ g_1,\dots,g_N \}
\]
and for $f\in\mathcal{G}$, let $f^{\sharp}$ denote the number of roots that $f$ has on $I$, and $\mathcal{G}^\sharp=\sup_{f\in \mathcal{G}\setminus\{\mathbf{0}\}}f^{\sharp}$.
Our first result is straightforward.

\begin{lem}\label{alternant_inverse}
Let $N\in\mathbb{N}$ and suppose that $G$ satisfies $\mathcal{G}^\sharp<N$ and $X=(x_i:1\leq i\leq N)\subset I$ consists of $N$ distinct points.  Then the alternant matrix
\[
A(G,X):=\left[ g_j(x_i)  \right]_{1\leq i,j\leq N}
\]
is invertible.
\end{lem}

\begin{proof}
Consider the product $A(G,X)\mathbf{a}$ in the variable $\mathbf{a}$.  This results in the vector $\mathbf{v}\in\mathbb{R}^N$, whose $i$-th component is given by $f(x_i)$, where $f\in\mathcal{G}$.  Now suppose that $A(G,X)$ is non-invertible.  Then the homogeneous system $A(G,X)\mathbf{a}=\mathbf{0}$ has a non trivial solution $\mathbf{a}_0$, which leads to $f_0\in\mathcal{G}$ that has $N$ roots on $I$.  This contradicts the fact that $\mathcal{M}<N$, which shows that $A(G,X)$ must be invertible.  
\end{proof}

The following straightforward lemma may help us calculate $\mathcal{G}^\sharp$.

\begin{lem}\label{derivative_roots}
Suppose that $f\in C^1(I)$ and that $f'$ has $N$ distinct roots in $I$.  Then $f$ has at most $N+1$ roots in $I$.
\end{lem}

\begin{proof}
We partition $I$ into $N+1$ subintervals with the roots of $f'$ as endpoints.  Since $f\in C^1(I)$, $f$ is monotone on each subinterval, so that there are at most $N+1$ roots of $f$.  
\end{proof}

So our problem is reduced to showing that if
\[
G:=\{ x^{-2},\dots,x^{-N},x^{-1}\ln(x),\dots,x^{-N}\ln(x)\}
\]
satisfies $\mathcal{G}^{\sharp}<2N-1$.  We will show that the set $\mathcal{H}=\{p(x)+q(x)\ln(x): p\in\Pi_{N-2}, q\in\Pi_{N-1} \}$ satisfies $\mathcal{H}^\sharp<2N-1$, which implies the bound for $\mathcal{G}^\sharp$.  Our argument makes use of the following derivative formulas, which may be easily verified via induction:
\begin{equation}\label{log derivative 1}
D^{k}\left( x^k\ln(x) \right) = k!\ln(x)+C_k; \quad k\in\mathbb{N} 
\end{equation}
for some positive constant $C_k$, and

\begin{equation}\label{log derivative 2}
D^{k+1}\left( x^k\ln(x) \right) = k!x^{-1}; \quad k\in\mathbb{N}_0.
\end{equation}
We'd like to include some information about what happens for a generic polynomial.  Our next result shows that we get an alternating combination of coefficients back.
\begin{lem}\label{lemma log derivative}
Let $N\geq 2$ and suppose that $p\in\Pi_{N-1}$, with 
\[
p(x)=\sum_{k=0}^{N-1}a_k x^k.
\]
Then
\[
D^{N} \left( p(x)\ln(x) \right) = x^{-N}\sum_{j=0}^{N-1}(-1)^{N-1+j}c_j a_j x^{j},
\]
for some positive constants $c_j$.
\end{lem}

\begin{proof}
We induct on $N\geq 2$.  Two applications of the product rule yields the base case:
\[
\left( (ax+b)\ln(x)\right)'' = \dfrac{ax-b}{x^2}.
\]
Now we assume that the conclusion holds for all $k$ with $2\leq k\leq N$.  Consider $p\in\Pi_{N}$. We have
\[
p(x)\ln(x)=\left(a_{N}x^{N}+q(x)\right)\ln(x),
\]
so that
\begin{align*}
D^{N+1} \left( p(x)\ln(x) \right)  &= a_N D^{N} \left( x^N\ln(x) \right) + D^{N+1} \left( q(x)\ln(x) \right) \\
&=a_N D^{N+1} \left( x^N\ln(x) \right) + D\left(x^{-N}\sum_{j=0}^{N-1}(-1)^{N-1+j}c_j a_j x^{j}\right)\\
&=N!a_Nx^{-1}+\sum_{j=0}^{N-1}(-1)^{N-1+j}c_j(j-N) a_j x^{j-N-1}\\
&=N!a_Nx^{-1}+\sum_{j=0}^{N-1}(-1)^{N+j}c_j(N-j) a_j x^{j-N-1}\\
&=x^{-N-1}\left( N!a_Nx^{N}+\sum_{j=0}^{N-1}(-1)^{N+j}c_j(N-j) a_j x^{j-N-1}  \right)\\
&=x^{-N-1}\sum_{j=0}^{N}(-1)^{N+j}\tilde{c}_ja_j x^j
\end{align*}
We've used \eqref{log derivative 2} in the third line.  The result follows from the fact that $c_j>0$ and $N-j>0$, so that $\tilde{c}_j>0$. 
\end{proof}

Suppose that $f(x) = p(x)+\sum_{j=0}^{N-1}a_j x^j\ln x \in \mathcal{H}$.  We would like to count the roots of $D^N f$, then repeatedly use Lemma \ref{derivative_roots} to count the roots of $f$.  Using \eqref{log derivative 2} and applying Descartes's rule of signs, shows that $D^n f$ has at most $N-1$ positive roots only when all of the coefficients share the same sign.  Lemma \ref{lemma log derivative} then would allow us to conclude that $f$ has at most $2N-1$ roots.  However, we can improve this bound under the assumption that all of the coefficients have the same sign.  Assuming all are positive, we have from \eqref{log derivative 1}
\begin{align*}
D^{N-1} f(x) =& D^{N-1}\left( \ln x \sum_{j=0}^{N-1}a_j x^j \right) \\
=&(N-1)!a_{N-1}\ln x + C_{N-1}  +x^{-N+1}\sum_{j=0}^{N-2}(-1)^{N+j}c_j a_j x^{j}\\
=&x^{-N+1}\left( x^{N-1}((N-1)!a_{N-1}\ln x +  C_{N-1})+\sum_{j=0}^{N-2}(-1)^{N+j}c_j a_j x^{j} \right)\\
=&:x^{-N+1}g(x).
\end{align*}
Now we can see that $D^{N-2}g(x) > 0$, so $D^{N-1}f$ has at most $N-2$ roots, hence $f$ has at most $2N-3$ roots in this case.  The same would be true if we took all of the coefficients negative.  Now we suppose that there is a sign change in the coefficients, then using Lemma \ref{lemma log derivative} and Descartes's rule of signs again provide at most $N-2$ roots for $D^N f$, so Lemma \ref{derivative_roots} shows that $f$ has at most $2N-2$ roots.  This yields $\mathcal{H}^{\sharp}\leq 2N-2$.  This shows that for $N\geq 2$, we can solve the system \eqref{log system} to isolate $A_N(x)$.  

Now we establish a bound for the growth rate of the solution components $\tilde{a}_i$.  Using Cramer's rule, we find an upper bound for the cofactor $A_i$ in the numerator and a lower bound for the determinant of $A$.  For a doubling sequence $(y_j)$, we have for positive contants $\alpha,\beta,$ and $\gamma$
\[
 \alpha y_{i}^{-(N^2+N-1)}  \leq |\det A| \leq \beta y_{i}^{-(N^2+N-1)}\left( \ln y_i \right)^{N}
\]
and
\[
|\det A_i| \leq \gamma y_i^{-(N^2-1)}\left(\ln y_i \right)^{N-1},
\]
so that 
\begin{equation}\label{log coefficient bound}
|\tilde{a}_i|\leq C y_{i}^{N}\left(\ln y_i\right)^{N-1} .
\end{equation}
If there is additional structure, for instance $X=\mathbb{Z}$, then we can get a sharper bound, but for our purposes, this isn't necessary.  Since $(A_k:k\in\mathbb{N})$ form a basis for $\Pi$, we have all of the necessary tools to prove a version of Theorem \ref{main 1}, which we write as a proposition. 

\begin{prop}
 Suppose that $X$ is a scattered sequence and that $\varphi(x) = x^{-1}\ln(1+x^2)$.  For any $f\in C[a,b]$ and $\varepsilon>0$, there exists $s\in S(\varphi,X)$ such that 
 \[
 \| f-s \|_{L_\infty}<\varepsilon.
 \]
\end{prop}

\begin{proof}
The proof is nearly identical to the one given for Theorem \ref{main 1} provided that $N\geq 2$, since \eqref{log coefficient bound} depends on $N$, we will be a bit more deliberate.  We note that since we first choose a polynomial $p$ using the Stone-Weierstrass theorem, we can fix $N=\deg(p)$.  Now just as before, we recover $(A_j:1\leq j\leq N)$.  We may recover $A_1$ with the $1\times 1$ matrix, which produces an error term that is $O(1/\ln(y_1))$.  For $j\geq 2$, the corresponding error term is $O(y_1^{-1}\ln(y_1)^j)$ rather than $O(y_1^{-1})$.  This means that there exists $(a_j)$ such that
\begin{align*}
\sum_{j=1}^{N}a_j\varphi(x-y_j) - p(x) &= O\left(\dfrac{1}{\ln(y_1)}\right)+O\left(\sum_{j=2}^{N}\dfrac{\ln(y_1)^{j}}{y_1} \right)\\
&= O\left(\dfrac{1}{\ln(y_1)}\right).
\end{align*}
Hence we can now choose $y_1$ so large that the error term is as small as we like.

\end{proof}

\subsection{Related Products}
The examples above can be multiplied to produce new examples.  Each of these examples is handled using the above results together with the Cauchy product similar to Lemma \ref{log lemma}.  For convenience, we simply state the results without providing the tedious details.

Let $L\in\mathbb{N}$, $q\in 2\mathbb{N}$ and suppose $qr\notin\mathbb{N}_0$, define the polynomials $(A_j)$, $(B_j)$, and $(C_j)$ by

\begin{align}\label{general related}
    (x-y)^{-L}\ln(1+(x-y)^q) &=: \ln(y)\sum_{j=0}^{\infty}A_j(x)y^{-j}+\sum_{j=0}^{\infty}B_j(x)y^{-j} \\
\label{arctan general}
\text{and }\arctan(x-y)\left(1+(x-y)^q\right)^r & =:y^{qr}\sum_{j=0}^{\infty}C_j(x)y^{-j}.
\end{align}

The following lemmas shows that we still get bases for large enough values of the index.

\begin{lem}
For $C_j$ defined in \eqref{arctan general}, we have
\[
    C_j(x) =(-1)^{j}\binom{qr-1}{j}x^{j-1}+ \emph{lower order terms};     \quad j\geq 1.
\]
\end{lem}

\begin{lem}
For $A_j$ and $B_j$ defined in \eqref{general related}, we have
\begin{align*}
    A_j(x) &=(-1)^{j+L}\binom{j-1}{L-1}qx^{j-L}+ \emph{lower order terms};     \quad j\geq L\\
    B_j(x) &= \sum_{i=1}^{j-L}(-1)^{L+1}\binom{j-1-i}{L-1}\left(\dfrac{q}{i}\right)x^{j-L}  +\emph{lower order terms};   \quad j\geq L+1.
\end{align*}
\end{lem}

Hence one may adjust the argument given in Theorem \ref{main 2}, \emph{mutatis mutandis}.

\end{document}